\documentclass[12pt,a4paper]{article}
\usepackage{amsfonts}
\usepackage{amsmath}
\usepackage{amssymb}
\usepackage[english]{babel}
\usepackage[T1]{fontenc}
\usepackage{graphicx}
\usepackage{subfigure}
\usepackage{epsfig}
\usepackage{latexsym}
\usepackage{amstext}
\usepackage{amsthm}
\usepackage{graphics}
\usepackage{enumerate}
\usepackage{fullpage}
\newtheorem{prop}{Proposition}[section]
\newtheorem{rem}[prop]{Remark}
\newtheorem{lem}[prop]{Lemma}
\newtheorem{theo}[prop]{Theorem}
\newtheorem{cor}[prop]{Corollary}

\newcommand{\beq}{\begin{eqnarray}}
\newcommand{\beqq}{\begin{eqnarray*}}
\newcommand{\eeq}{\end{eqnarray}}
\newcommand{\eeqq}{\end{eqnarray*}}

\def\QED{\quad\hbox{\hskip 4pt\vrule width 5pt height 6pt depth 1.5pt}}
\title{Windings of planar stable processes}
\author{ R.A. Doney \thanks{Probability and Statistics Group, School of Mathematics, University of Manchester,
Alan Turing Building, Oxford Road, Manchester M13 9PL, United Kingdom. E-mail: ron.doney@manchester.ac.uk} \ and \
S. Vakeroudis $^{\ast}$\thanks{Laboratoire de Probabilit\'{e}s et Mod\`{e}les
Al\'{e}atoires (LPMA) CNRS : UMR7599,  Universit\'{e} Pierre et Marie
Curie - Paris VI,  Universit\'{e} Paris-Diderot - Paris VII, 4 Place Jussieu, 75252 Paris Cedex 05, France.
E-mail: stavros.vakeroudis@ulb.ac.be } }
\date{\today}
\begin{document}
%%%%%%%%%%%%%%%%%%%%%%%%%%%%%%%%%%%%%%%%%%%%%%%%%%%%%%%%%%%%%%%%

%%%%%%%%%%%%%%%%%%%%%%%%%%%%%%%%%%%%%%%%%%%%%%%%%%%%%%%%%%%%%%%%
\maketitle
\begin{abstract}
Using a generalization of the skew-product representation of planar Brownian motion
and the analogue of Spitzer's celebrated asymptotic Theorem
for stable processes due to Bertoin and Werner, for which we provide a new easy proof,
we obtain some limit Theorems for the exit time from a cone of stable processes
of index $\alpha\in(0,2)$. We also study the case $t\rightarrow0$
and we prove some Laws of the Iterated Logarithm (LIL)
for the (well-defined) winding process associated to our planar stable process.
\end{abstract}
%%%%%%%%%%%%%%%%%%%%%%%%%%%%%%%%%%%%%%%%%%%%%%%%%%%%%%%%%%%%%%%%

%%%%%%%%%%%%%%%%%%%%%%%%%%%%%%%%%%%%%%%%%%%%%%%%%%%%%%%%%%%%%%%%
$\vspace{5pt}$
%%%%%%%%%%%%%%%%%%%%%%%%%%%%%%%%%%%%%%%%%%%%%%%%%%%%%%%%%%%%%%%%

%%%%%%%%%%%%%%%%%%%%%%%%%%%%%%%%%%%%%%%%%%%%%%%%%%%%%%%%%%%%%%%%
$\vspace{5pt}$

%%%%%%%%%%%%%%%%%%%%%%%%%%%%%%%%%%%%%%%%%%%%%%%%%%%%%%%%%%%%%%%%
\section{Introduction}
%%%%%%%%%%%%%%%%%%%%%%%%%%%%%%%%%%%%%%%%%%%%%%%%%%%%%%%%%%%%%%%%
\renewcommand{\thefootnote}{\fnsymbol{footnote}}
In this paper, we study the windings of planar isotropic stable processes.
More precisely, having as a starting point a work of Bertoin and Werner \cite{BeW96} concerning this subject
(following their previous work on windings of planar Brownian motion\footnote[4]{When we simply write: Brownian motion, we always mean real-valued Brownian motion, starting from 0. For 2-dimensional Brownian motion, we indicate planar or complex BM.} \cite{BeW94})
and motivated by some works of Shi \cite{Shi98}, we attempt to generalize some results obtained recently for the case
of planar Brownian motion (see e.g. \cite{Vak11,Vakth11,VaY11a} and the references therein).
In particular, we are interested in the behaviour of stable processes for small time,
an aspect which has already been investigated e.g. by Doney \cite{Don04} in terms
of Spitzer's condition for stable processes (see e.g. \cite{BeD97} and the references therein).

In Section \ref{secpre}, we recall some facts about standard isotropic
stable processes of index $\alpha\in(0,2)$ taking values in the complex plane.
Then, we follow Bertoin and Werner \cite{BeW96} to define the process of its winding number, we
generalize the skew-product representation of planar BM (see e.g. \cite{Kiu80,ReY99,Chy06}) and
we present two Lemmas for the winding process of isotropic stable L\'{e}vy processes obtained in \cite{BeW96}.
Finally, we mention some properties of the positive and the negative moments of the exit times from a cone of this process.

In Section \ref{larget}, we use some continuity arguments of the composition function due to Whitt \cite{Whi80} and we obtain a new simple proof
of the analogue of Spitzer's Theorem for isotropic stable L\'{e}vy processes,
initially proven by Bertoin and Werner \cite{BeW96}. We reformulate and we extend this result in terms of the exit times from a cone.
More precisely, Spitzer's asymptotic Theorem says that, if $(\vartheta_{t},t\geq0)$
denotes the continuous determination of the argument of a planar BM starting away from the origin, then:
\beq\label{Spi}
 \frac{2}{\log t} \; \vartheta_{t} \overset{{(law)}}{\underset{t\rightarrow\infty}\longrightarrow} C_{1} \ ,
\eeq
where $C_{1}$ is a standard Cauchy variable. For other proofs of (\ref{Spi}) , see e.g.
\cite{Wil74,Dur82,MeY82,PiY86,BeW94,Yor97,Vak11,VaY11a}.
Bertoin and Werner state that because an isotropic stable L\'{e}vy processes is transient,
we expect that it winds more slowly than planar Brownian motion and prove that, with $\theta$ now denoting the
process of its winding number (appropriately defined, see e.g. Section \ref{secpre}),
$\theta_{t}/\sqrt{\log t}$ converges in
distribution to some centered Gaussian law as $t\rightarrow\infty$
(Theorem 1 in Bertoin and Werner \cite{BeW96}, stated here as Theorem \ref{SpiBW}).

In Section \ref{smallt}, and more precisely in Theorems \ref{DV} and \ref{DV2},
we study the asymptotics of a symmetric L\'{e}vy process and of the winding process of isotropic stable
L\'{e}vy processes for $t\rightarrow0$, respectively, which are the main results of this article.
In particular, we show that $t^{-1/\alpha}\theta_{t}$ converges in
distribution to an $\alpha$-stable law as $t\rightarrow0$.
Using this result, in Proposition \ref{proptsmall} we obtain the (weak) limit in distribution of the process of the exit times
from a cone with narrow amplitude and we further obtain several generalizations.
We also study the windings of planar stable processes in $\left(\right.t,1\left.\right]$, for $t\rightarrow0$ and
we note that, with obvious notation, Spitzer's law is still valid for $\theta_{\left(\right.t,1\left.\right]}$ (see Remark \ref{remt1}).
Section \ref{secLIL} deals with the Law of the Iterated Logarithm LIL for L\'{e}vy processes for small times,
in the spirit of some well-known (LIL) for Brownian motion for $t\rightarrow\infty$
from Bertoin and Werner \cite{BeW94,BeW94b} and from Shi \cite{Shi94,Shi98}, and for stable subordinators
with index $\alpha\in(0,1)$ for $t\rightarrow0$ from Fristedt \cite{Fri64,Fri67} and Khintchine \cite{Khi39}
(see also \cite{Ber96}). Moreover, we prove a LIL for the winding number process of stable processes,
for $t\rightarrow0$.

Finally, in Section \ref{secBM} we discuss the planar Brownian motion case and in Theorem \ref{DV3} we obtain
the asymptotic behaviour of the winding process as $t\rightarrow0$. More precisely,
the process $\left(c^{-1/2}\vartheta_{ct},t\geq0\right)$ converges in law to a 1-dimensional Brownian motion
as $c\rightarrow0$.
\\ \\
\noindent\textbf{Notation:} In the following text, with the symbol "$\Longrightarrow$"
we shall denote the weak convergence in distribution on the appropriate space, endowed with
the Skorohod topology.

%%%%%%%%%%%%%%%%%%%%%%%%%%%%%%%%%%%%%%%%%%%%%%%%%%%%%%%%%%%%%%%%
\section{Preliminaries}\label{secpre}
%%%%%%%%%%%%%%%%%%%%%%%%%%%%%%%%%%%%%%%%%%%%%%%%%%%%%%%%%%%%%%%%
Following Lamperti \cite{Lam72}, a Markov process $X$ with values in $\mathbb{R}^{d}$, $d\geq2$
is called isotropic or $O(d)$-invariant ($O(d)$ stands for the group of orthogonal transformations
on $\mathbb{R}^{d}$) if its transition satisfies:
\beq
P_{t}(\phi(x),\phi(\mathcal{B}))=P_{t}(x,\mathcal{B}),
\eeq
for any $\phi\in O(d)$, $x\in \mathbb{R}^{d}$ and Borel subset $\mathcal{B}\subset\mathbb{R}^{d}$. \\
Moreover, $X$ is said to be $\alpha$-self-similar if, for $\alpha>0$,
\beq\label{ass}
P_{\lambda t}(x,\mathcal{B})=P_{t}(\lambda^{-\alpha}x,\lambda^{-\alpha}\mathcal{B}),
\eeq
for any $\lambda>0$, $x\in \mathbb{R}^{d}$ and $\mathcal{B}\subset\mathbb{R}^{d}$.

We focus now our study on the 2-dimensional case $(d=2)$, where (\ref{ass}) holds, and we denote by
$(Z_{t},t\geq0)$ a standard isotropic stable process of index $\alpha\in(0,2)$ taking values in the complex plane
and starting from $z_{0}+i0,z_{0}>0$.
A scaling argument shows that we may assume $z_{0}=1$,
without loss of generality, since, with obvious notation:
\beq
\left(Z^{(z_{0})}_{t},t\geq0\right)\stackrel{(law)}{=}\left(z_{0}Z^{(1)}_{(t/z^{\alpha}_{0})},t\geq0\right).
\eeq
Thus, from now on, we shall take $z_{0}=1$.
More precisely, $Z$ has stationary independent increments,
its sample path is right continuous and has left limits (cadlag) and, with $\langle\cdot,\cdot\rangle$ standing for the Euclidean inner product, $E\left[\exp\left(i\langle\lambda,Z_{t}\rangle\right)\right]=\exp\left(-t |\lambda|^{\alpha}\right)$, for all $t\geq0$ and $\lambda\in\mathbb{C}$.
$Z$ is transient, $\lim_{t\rightarrow\infty} |Z_{t}|=\infty$ a.s.
and it a.s. never visits single points.
We remark that for $\alpha=2$, we are in the Brownian motion case.

We are now going to recall some properties of stable processes and L\'{e}vy processes
(for more details see e.g. \cite{Ber96} or \cite{Kyp06}). \\
To start with, if $\mathcal{Z}=(\mathcal{Z}_{t},t\geq0)$ denotes a planar Brownian motion starting from 1 and
$S=(S(t),t\geq0)$ an independent stable subordinator with index $\alpha/2$ starting from 0, i.e.:
\beq\label{subordinator}
E\left[\exp\left(-\mu S(t)\right)\right]=\exp\left(-t \mu^{\alpha/2}\right),
\eeq
for all $t\geq0$ and $\mu\geq0$, then the subordinated planar BM $(\mathcal{Z}_{2S(t)},t\geq0)$
is a standard isotropic stable process of index $\alpha$.
The L\'{e}vy measure of $S$ is:
$$ \frac{\alpha}{2\Gamma(1-\alpha/2)} s^{-1-\alpha/2} 1_{\{s>0\}} ds $$
thus, the L\'{e}vy measure $\nu$ of $Z$ is:
\beq
\nu(dx)&=& \frac{\alpha}{2\Gamma(1-\alpha/2)} \int^{\infty}_{0} s^{-1-\alpha/2} P\left(\mathcal{Z}_{2s}-1 \ \in \ dx\right) ds \nonumber \\
&=& \frac{\alpha}{8\pi\Gamma(1-\alpha/2)} \left(\int^{\infty}_{0} s^{-2-\alpha/2} \exp\left(-|x|^{2}/(4s)\right) \ ds\right) dx \nonumber \\
&=&\frac{\alpha \ 2^{-1+\alpha/2} \Gamma(1+\alpha/2)}{\pi\Gamma(1-\alpha/2)} |x|^{-2-\alpha} dx \ .
\eeq

Contrary to planar Brownian motion, as $Z$ is discontinuous, we cannot define its winding number
(recall that, as is well known \cite{ItMK65}, for planar BM, since
it starts away from the origin, it does not visit a.s. the point $0$
but keeps winding around it infinitely often. In particular, the
winding process is well defined, for further details see also e.g. \cite{PiY86}).
However, following \cite{BeW96}, we can consider a path on a finite time
interval $[0,t]$ and "fill in" the gaps with line segments in order to obtain the curve of a continuous function
$f:[0,1]\rightarrow\mathbb{C}$ with $f(0)=1$. Now, since 0 is polar and $Z$ has no jumps across 0 a.s., we have $f(u)\neq0$
for every $u\in[0,1]$. Hence, we can define the process of the winding number of $Z$ around 0, which
we denote by $\theta=(\theta_{t},t\geq0)$. It has cadlag paths of absolute length greater than $\pi$ and, for all $t\geq0$,
\beq
\exp(i\theta_{t})=\frac{Z_{t}}{|Z_{t}|} \ .
\eeq
We also introduce the clock:
\beq\label{basicclock}
H(t)\equiv\int^{t}_{0}\frac{ds}{\left|Z_{s}\right|^{\alpha}} \ ,
\eeq
and its inverse:
\beq
A(u)\equiv\inf\{t\geq0, H(t)>u\} \ .
\eeq
Bertoin and Werner following \cite{GVA86} obtained these two Lemmas for $\alpha\in(0,2)$ (for the proofs see \cite{BeW94}):
%%%%%%%%%%%%%%%%%%%%%%%%%%%%%%%%%%%%%%%%%%%%%%%%%%%%%%%%%%%%%%%%
\begin{lem}\label{lemma1}
The time-changed process $(\theta_{A(u)},u\geq0)$ is a real-valued symmetric L\'{e}vy process.
It has no Gaussian component and its L\'{e}vy measure has support in $[-\pi,\pi]$.
\end{lem}
%%%%%%%%%%%%%%%%%%%%%%%%%%%%%%%%%%%%%%%%%%%%%%%%%%%%%%%%%%%%%%%%
We now denote by $dz$ the Lebesgue measure on $\mathbb{C}$. Then, for every complex number $z\neq0$, $\phi(z)$ denotes
the determination of its argument valued in $\left(\right.-\pi,\pi\left.\right]$.
%%%%%%%%%%%%%%%%%%%%%%%%%%%%%%%%%%%%%%%%%%%%%%%%%%%%%%%%%%%%%%%%
\begin{lem}\label{lemma2}
The L\'{e}vy measure of $\theta_{A(\cdot)}$ is the image of the L\'{e}vy measure $\nu$ of $Z$
by the mapping $z\rightarrow\phi(1+z)$. As a consequence, $E[(\theta_{A(u)})^{2}]=u k(\alpha)$, where
\beq\label{constalpha}
k(\alpha) = \frac{\alpha \ 2^{-1+\alpha/2} \Gamma(1+\alpha/2)}{\pi\Gamma(1-\alpha/2)} \int_{\mathbb{C}} |z|^{-2-\alpha} |\phi(1+z)|^{2} dz \ .
\eeq
\end{lem}
%%%%%%%%%%%%%%%%%%%%%%%%%%%%%%%%%%%%%%%%%%%%%%%%%%%%%%%%%%%%%%%%
Using Lemma \ref{lemma1}, we can obtain the analogue of the skew product representation for planar BM
which is the Lamperti correspondence for stable processes.
Indeed, following \cite{GVA86} and using Lamperti's relation (see e.g. \cite{Lam72,Kiu80,Chy06,CPP11} or \cite{ReY99}) and Lemma \ref{lemma1},
there exist two real-valued L\'{e}vy processes $(\xi_{u},u\geq0)$ and $(\rho_{u},u\geq0)$, the first one non-symmetric
whereas the second one symmetric, both starting from 0, such that:
\beq\label{skew-product}
\log\left|Z_{t}\right|+i\theta_{t}
=\left(\xi_{u}+i\rho_{u}\right)
\Bigm|_{u=H_{t}=\int^{t}_{0}\frac{ds}{\left|Z_{s}\right|^{\alpha}}} \ .
\eeq
We remark here that $|Z|$ and $Z_{A(\cdot)}/|Z_{A(\cdot)}|$ are NOT independent. Indeed, the processes
$|Z_{A(\cdot)}|$ and $Z_{A(\cdot)}/|Z_{A(\cdot)}|$ jump at the same times hence they cannot be independent.
Moreover, $A(\cdot)$ depends only upon $|Z|$, hence $|Z|$ and $Z_{A(\cdot)}/|Z_{A(\cdot)}|$ are not independent.
For further discussion on the independence, see e.g. \cite{LiW11}, where is shown that an isotropic
$\alpha$-self-similar Markov process has a skew-product structure if and only if its radial and its
angular part do not jump at the same time.

We also remark that
\beq\label{HAclock}
H^{-1}(u)\equiv A(u)\equiv \inf\{t\geq0: H(t)>u\}=\int^{u}_{0} \exp\{\alpha \xi_{s}\} \ ds \ .
\eeq
Hence, (\ref{skew-product}) may be equivalently written as:
\beq\label{skew-product2}
\left\{
  \begin{array}{ll}
    \left|Z_{t}\right|=\exp\left(\xi(H_{t})\right)\Leftrightarrow\left|Z_{A(t)}\right|=\exp\left(\xi_{t}\right), & \hbox{(extension of Lamperti's identity)} \\
    \theta_{t}=\rho(H_{t})\Leftrightarrow \theta\left(A(t)\right)=\rho(t) \ .
  \end{array}
\right.
\eeq
We also define the random times $T^{|\theta|}_{c}\equiv\inf\{ t:|\theta_{t}|\geq c \}$ and
$T^{|\rho|}_{c}\equiv\inf\{ t:|\rho_{t}|\geq c \}$, $(c>0)$.
Using the "generalized" skew-product representation (\ref{skew-product}) (or (\ref{skew-product2})), we obtain:
\beq \label{skew-productplanar}
T^{|\theta|}_{c}=H^{-1}_{u}\Bigm|_{u=T^{|\rho|}_{c}}=\int^{T^{|\gamma|}_{c}}_{0}ds\exp(\alpha\xi_{s})\equiv A_{T^{|\rho|}_{c}} \ .
\eeq
Following \cite{VaY11a}, for the random times $T^{\theta}_{-d,c}\equiv\inf\{ t:\theta_{t}\notin(-d,c) \}$, $d,c>0$,
and $T^{\theta}_{c}\equiv\inf\{ t:\theta_{t}\geq c \}$, we have:
%%%%%%%%%%%%%%%%%%%%%%%%%%%%%%%%%%%%%%%%%%%%%%%%%%%%%%%%%%%%%%%%%%
\begin{rem} \label{rem}
For $0<c<d$, the random times $T^{\theta}_{-d,c}$, $T^{|\theta|}_{c}$ and $T^{\theta}_{c}$
satisfy the trivial inequality:
\beq
T^{|\theta|}_{c}\leq T^{\theta}_{-d,c} \leq T^{\theta}_{c} .
\eeq
Hence, with $p>0$:
\beq
E\left[\left(T^{|\theta|}_{c}\right)^{p}\right]\leq E\left[\left(T^{\theta}_{-d,c}\right)^{p}\right]
\leq E\left[\left(T^{\theta}_{c}\right)^{p}\right] \ ,
\eeq
and for the negative moments:
\beq
E\left[\left(T^{\theta}_{c}\right)^{-p}\right]\leq E\left[\left(T^{\theta}_{-d,c}\right)^{-p}\right]
\leq E\left[\left(T^{|\theta|}_{c}\right)^{-p}\right] \ .
\eeq
\end{rem}
%%%%%%%%%%%%%%%%%%%%%%%%%%%%%%%%%%%%%%%%%%%%%%%%%%%%%%%%%%%%%%%%%%
\begin{rem} \label{rem1}
For further details concerning the finiteness of the positive moments of
$T^{|\theta|}_{c}$, see e.g. \cite{DeB90,BaB04}.
Recall also that for the positive moments of the exit time from a cone of planar Brownian motion,
Spitzer showed that (with obvious notation) \cite{Spi58,Bur77}:
\beq
E\left[\left(T^{|\vartheta|}_{c}\right)^{p}\right]<\infty \Leftrightarrow p<\frac{\pi}{4c} \ ,
\eeq
whereas all the negative moments $E\left[\left(T^{|\vartheta|}_{c}\right)^{-p}\right]$ are finite \cite{VaY11a}.
\end{rem}
%%%%%%%%%%%%%%%%%%%%%%%%%%%%%%%%%%%%%%%%%%%%%%%%%%%%%%%%%%%%%%%%%%
We denote now by $\Psi(u)$ the exponent of the symmetric L\'{e}vy process $\rho$, hence (L\'{e}vy-Khintchine formula)
$E[e^{iu \rho_{t}}]=e^{-t\Psi(u)}$, with:
\beq\label{lm}
    \Psi(u) = \int_{(-\infty,\infty)}\left(1-e^{iux}+iux 1_{\{|x|\leq1\}}\right) \mu(dx), \ u\in\mathbb{R},
\eeq
where $\mu$ is a Radon measure on $\mathbb{R}\setminus \{0\}$ such that:
\beqq
    \int_{(-\infty,\infty)} (x^{2}\wedge1) \mu(dx)< \infty \ .
\eeqq
$\mu$ is the L\'{e}vy measure of $\rho$ and is symmetric.

%%%%%%%%%%%%%%%%%%%%%%%%%%%%%%%%%%%%%%%%%%%%%%%%%%%%%%%%%%%%%%%%
\section{Large time asymptotics}\label{larget}
%%%%%%%%%%%%%%%%%%%%%%%%%%%%%%%%%%%%%%%%%%%%%%%%%%%%%%%%%%%%%%%%
Concerning the clock $H$,
we have the almost sure convergence (see Corollary 1 in Bertoin and Werner \cite{BeW96}):
\beq\label{clocktlarge}
\frac{H(e^{u})}{u}  &\overset{{a.s.}}{\underset{u\rightarrow\infty}\longrightarrow}& 2^{-\alpha} \frac{\Gamma(1-\alpha/2)}{\Gamma(1+\alpha/2)}
\equiv K(\alpha) = E\left[|Z_{1}|^{-\alpha}\right] \ .
\eeq
Moreover, we have the following:
%%%%%%%%%%%%%%%%%%%%%%%%%%%%%%%%%%%%%%%%%%%%%%%%%%%%%%%%%%%%%%%%
\begin{prop}\label{proptightness}
The family of processes
$$H^{(u)}_{x}\equiv\left(\frac{H(e^{ux})}{u}, x\geq0\right)$$ is tight, as $u\rightarrow \infty$.
\end{prop}
%%%%%%%%%%%%%%%%%%%%%%%%%%%%%%%%%%%%%%%%%%%%%%%%%%%%%%%%%%%%%%%%
%%%%%%%%%%%%%%%%%%%%%%%%%%%%%%%%%%%%%%%%%%%%%%%%%%%%%%%%%%%%%%%%
\noindent{\textbf{Proof of Proposition \ref{proptightness}:}} \\
To prove this, we could repeat some arguments of Pitman and Yor \cite{PiY89}
(see the estimates in their proof of Theorem 6.4), however,
we give here a straightforward proof, using the definition of tightness: \\
for every $\varepsilon,\eta>0$, there exist $\delta>0$ and $C_{\delta}>0$ such that, for every $0<x<y$:
\beq\label{tightnessH}
P\left(\sup_{|x-y|\leq \delta}\left|H(e^{uy})-H(e^{ux})\right|\geq u \varepsilon\right)\leq \eta \ , \ \ \mathrm{for} \ \ u\geq C_{\delta} \ ,
\eeq
or equivalently:
\beq\label{tightnessH2}
P\left(\frac{1}{u} \ \left|H(e^{u(x+\delta)})-H(e^{ux})\right|\geq \varepsilon\right)\leq \eta \ , \ \ \mathrm{for} \ \ u\geq C_{\delta} \ .
\eeq
First, following Bertoin and Werner \cite{BeW96}, we introduce the "Ornstein-Uhlenbeck type" process:
\beq\label{defZtilde}
\tilde{Z}_{u}=\exp(-u/\alpha)Z_{\exp(u)}, \ \ \ u\geq0 \ ,
\eeq
which is a stationary Markov process under $P_{0}$ (see e.g. \cite{Bre68}). We denote by $p_{t}(\cdot)$ the semigroup of $Z$:
$$p_{t}(\bar{z})=P_{0}(Z_{t}\in d\bar{z})/d\bar{z}, \ \ \bar{z}\in\mathbb{C}.$$
We denote by $Z^{(0)}$ another stable process starting at 0.
Then, using the scaling property, given that $\tilde{Z}_{0}\equiv Z_{1}\equiv 1+Z^{(0)}_{1}=\bar{x}$,
the semigroup $q_{u}(\cdot)$ of $\tilde{Z}$ is given by:
\beq\label{semigr}
q_{u}(\bar{x},\bar{y})&=&p_{\exp(u)-1}\left(e^{u/\alpha}\bar{y}-\bar{x}\right)e^{2u/\alpha} \nonumber \\
&=& (e^{u}-1)^{-2/\alpha}e^{2u/\alpha} p_{1}\left((e^{u}-1)^{-1/\alpha}(e^{u/\alpha}\bar{y}-\bar{x})\right) \nonumber \\
&=& (l(u))^{2} p_{1}\left(l(u)(\bar{y}-e^{-u/\alpha}\bar{x})\right),
\eeq
where $l(v)=e^{v/\alpha}(e^{v}-1)^{-1/\alpha}$.
For every $\delta>0$ and changing variables $s=\exp(v)$, with obvious notation, we have:
\beq\label{tightproof}
E\left[\left|H(e^{u(x+\delta)})-H(e^{ux})\right|\right]=
\int^{e^{u(x+\delta)}}_{e^{ux}} E\left[|Z_{s}|^{-\alpha}\right] \ ds
=\int^{u(x+\delta)}_{ux} E_{\tilde{Z}_{0}}\left[|\tilde{Z}_{v}|^{-\alpha}\right] \ dv \ .
\eeq
We also define $\varepsilon(v)\equiv l(v)e^{-v/\alpha}=(e^{v}-1)^{-1/\alpha}$.
From (\ref{defZtilde}), using the stability of $Z$, we have:
\beq
\tilde{Z}_{v}&=&e^{-v/\alpha}Z_{\exp(v)}=e^{-v/\alpha}\left(Z_{\exp(v)-1}+Z_{1}\right)
\stackrel{(law)}{=}e^{-v/\alpha}\left((e^{v}-1)^{1/\alpha}Z^{(0)}_{1}+Z_{1}\right) \nonumber \\
&=& (l(v))^{-1}\left(Z^{(0)}_{1}+\varepsilon(v)Z_{1}\right).
\eeq
Hence (for simplicity, we use $E\equiv E_{0}$):
\beq\label{tightproof1}
E_{\tilde{Z}_{0}}\left[|\tilde{Z}_{v}|^{-\alpha}\right]=(l(v))^{\alpha} E\left[|Z^{(0)}_{1}+\varepsilon(v)\tilde{Z}_{0}|^{-\alpha}\right]
\equiv (l(v))^{\alpha} \left(E_{1}+E_{2}\right),
\eeq
where, with $\delta'>0$,
\beqq
E_{1}&=& E\left[|Z^{(0)}_{1}+\varepsilon(v)\tilde{Z}_{0}|^{-\alpha}:|Z^{(0)}_{1}+\varepsilon(v)\tilde{Z}_{0}|\geq\delta'\right] \ , \\
E_{2}&=& E\left[|Z^{(0)}_{1}+\varepsilon(v)\tilde{Z}_{0}|^{-\alpha}:|Z^{(0)}_{1}+\varepsilon(v)\tilde{Z}_{0}|\leq\delta'\right] \ .
\eeqq
We have: $l(v)\stackrel{v\rightarrow \infty}{\longrightarrow} 1$ and
$\varepsilon(v)\stackrel{v\rightarrow \infty}{\longrightarrow} 0$, thus, by Dominated Convergence Theorem:
\beq\label{tightproof2}
E_{1}\stackrel{v\rightarrow \infty}{\longrightarrow} E\left[|Z^{(0)}_{1}|^{-\alpha}:|Z^{(0)}_{1}|\geq\delta'\right]
\stackrel{\delta'\rightarrow 0}{\longrightarrow} E\left[|Z^{(0)}_{1}|^{-\alpha}\right] \ .
\eeq
Moreover, changing the variables: $\bar{w}=\bar{z}+\varepsilon(v)\bar{x}$, we have:
\beqq
E_{2}&=& \int_{\bar{x},\bar{z}:|\bar{z}+\varepsilon(v)\bar{x}|\leq\delta'}
P(\tilde{Z}_{0}\in d\bar{x}) \ P(Z^{(0)}_{1}\in d\bar{z}) \ |\bar{z}+\varepsilon(v)\bar{x}|^{-\alpha} \\
&=& \int_{\bar{x},\bar{w}:|\bar{w}|\leq\delta'}
P(\tilde{Z}_{0}\in d\bar{x}) \ P(Z^{(0)}_{1}\in d\bar{w}) \ |\bar{w}|^{-\alpha} .
\eeqq
Remarking now that for stable processes: $P(Z^{(0)}_{1}\in d\bar{y})\leq C' d\bar{y}$, where $C'$ stands for a positive constant
and using $w=(w_{1},w_{2})$, we have:
\beq\label{tightproof3}
E_{2}\leq C' \int_{\bar{x},\bar{w}:|\bar{w}|\leq\delta'}
P(\tilde{Z}_{0}\in d\bar{x}) \ \frac{dw_{1} \ dw_{2}}{|\bar{w}|^{\alpha}}=
C' \int_{\bar{z}:|\bar{w}|\leq\delta'}\frac{dw_{1} \ dw_{2}}{|\bar{w}|^{\alpha}} \stackrel{\delta'\rightarrow 0}{\longrightarrow} 0 \ .
\eeq
Thus, from (\ref{tightproof1}), (\ref{tightproof2}) and (\ref{tightproof3}), invoking again the Dominated
Convergence Theorem, we deduce:
\beq
\lim_{u\rightarrow\infty} E_{\tilde{Z}_{0}}\left[|\tilde{Z}_{v}|^{-\alpha}\right]=E\left[|Z^{(0)}_{1}|^{-\alpha}\right] \ ,
\eeq
which is a constant. Hence, for every $\varepsilon, \eta>0$, there exist $\delta>0$ and $C_{\delta}>0$ such that
(\ref{tightnessH2}) is satisfied for $u\geq C_{\delta}$.
\hfill \QED
%%%%%%%%%%%%%%%%%%%%%%%%%%%%%%%%%%%%%%%%%%%%%%%%%%%%%%%%%%%%%%%%
\\ \\
Bertoin and Werner in \cite{BeW96} obtained the analogue of Spitzer's asymptotic Theorem \cite{Spi58}
for isotropic stable L\'{e}vy processes of index $\alpha\in(0,2)$:
%%%%%%%%%%%%%%%%%%%%%%%%%%%%%%%%%%%%%%%%%%%%%%%%%%%%%%%%%%%%%%%%
\begin{theo}\label{SpiBW}
The family of processes $$\Theta^{(c)}_{t}\equiv\left(c^{-1/2}\theta_{\exp(ct)},t\geq0\right)$$
converges in distribution on $D(\left[\right.0,\infty\left.\right),\mathbb{R})$
endowed with the Skorohod topology, as $c\rightarrow\infty$, to $\left(\sqrt{r(\alpha)}B_{t},t\geq0\right)$,
where $\left(B_{s},s\geq0\right)$ is a real valued Brownian motion and
\beq\label{cst}
r(\alpha)=\frac{\alpha \ 2^{-1-\alpha/2}}{\pi} \int_{\mathbb{C}} |z|^{-2-\alpha} |\phi(1+z)|^{2} dz \ .
\eeq
\end{theo}
%%%%%%%%%%%%%%%%%%%%%%%%%%%%%%%%%%%%%%%%%%%%%%%%%%%%%%%%%%%%%%%%
\noindent{Using some results due to Whitt \cite{Whi80}, we can obtain a simple proof of this Theorem.}
\\ \\
%%%%%%%%%%%%%%%%%%%%%%%%%%%%%%%%%%%%%%%%%%%%%%%%%%%%%%%%%%%%%%%%
\textbf{Proof of Theorem \ref{SpiBW}: (new proof)} \\
Essentially, an argument of continuity of the composition function (Theorem 3.1 in \cite{Whi80}) may replace the martingale argument
in the lines of the proof for $t\rightarrow\infty$ from Bertoin and Werner.
We split the proof in three parts: \\ \\
$\left.\mathrm{i}\right)$ Concerning the clock $H$,
we have the almost sure convergence (\ref{clocktlarge}):
\beqq
\frac{H(e^{u})}{u}  &\overset{{a.s.}}{\underset{u\rightarrow\infty}\longrightarrow}& K(\alpha) = E\left[|Z_{1}|^{-\alpha}\right] \ .
\eeqq
From this result follows the convergence of the finite dimensional distributions of $v^{-1}H(\exp(vt))$,
for $v\rightarrow\infty$ and every $t>0$. \\
Moreover, from Proposition \ref{proptightness}, the family of processes
$$H^{(u)}_{x}\equiv\left(\frac{H(e^{ux})}{u}, x\geq0\right)$$ is tight as $u\rightarrow \infty$.
Hence, from (\ref{clocktlarge}) and (\ref{tightnessH}), finally, $H^{(u)}(t)\equiv(u^{-1}H(\exp(ut)),t\geq0)$
converges weakly to $(t K(\alpha), t\geq0)$ as $u\rightarrow\infty$, i.e.:
\beq\label{prlarge1}
\left(\frac{H(e^{ut})}{u},t\geq 0\right)\overset{{(d)}}{\underset{u\rightarrow\infty}\Longrightarrow}(t K(\alpha), t\geq0),
\eeq
where the convergence in distribution is viewed on $D(\left[\right.0,\infty\left.\right),\mathbb{R})$
endowed with the Skorohod topology. \\ \\
$\left.\mathrm{ii}\right)$
Using the skew product representation analogue (\ref{skew-product}) and Lemma \ref{lemma2}, we have:
\beq\label{prlarge2}
\left(\frac{\rho_{tu}}{\sqrt{u}},t\geq 0\right)\overset{{(d)}}{\underset{u\rightarrow\infty}\Longrightarrow}\left(\sqrt{k(\alpha)} \ B_{t},t\geq 0\right) \ ,
\eeq
where the convergence in distribution is viewed again on $D(\left[\right.0,\infty\left.\right),\mathbb{R})$
endowed with the Skorohod topology and $k(\alpha)$ is given by (\ref{constalpha}). This follows from the convergence
of the finite dimensional distributions:
\beq
\frac{\rho_{u}}{\sqrt{u}}=\frac{\theta_{A(u)}}{\sqrt{u}}\overset{{(d)}}{\underset{u\rightarrow\infty}\longrightarrow}\sqrt{k(\alpha)} \ B_{1} \ ,
\eeq
a condition which is sufficient for the weak convergence (\ref{prlarge2}), since L\'{e}vy processes are semimartingales with stationary independent
increments; for further details see e.g. \cite{Sko91} or \cite{JaS03} (Corollary 3.6, Chapter VII, p. 415). \\ \\
$\left.\mathrm{iii}\right)$
Theorem 3.1 in \cite{Whi80} states that the composition function on
$D(\left[\right.0,\infty\left.\right),\mathbb{R})\times D(\left[\right.0,\infty\left.\right),\left[\right.0,\infty\left.\right))$
is continuous at each $(\rho,H)\in(C(\left[\right.0,\infty\left.\right),\mathbb{R}) \times D_{0}(\left[\right.0,\infty\left.\right),\left[\right.0,\infty\left.\right)))$,
with $C$ denoting the set of continuous functions and $D_{0}$ the subset of increasing cadlag functions in $D$
(hence the subset of non-decreasing cadlag functions in $D$). Hence, from (\ref{prlarge1}) and (\ref{prlarge2}),
we have: for every $t>0$,
\beq\label{cvgprlarge}
\frac{\theta_{\exp(ct)}}{\sqrt{c}}= \frac{\rho_{H(\exp(ct))}}{\sqrt{c}}=\frac{\rho_{c(H(e^{ct})/c)}}{\sqrt{c}} \ .
\eeq
The result now follows from the continuity of the composition function
together with (\ref{cvgprlarge}) and the weak convergence of $H^{(c)}(\cdot)$ and $c^{-1/2}\rho_{c}$,
as $c\rightarrow\infty$.
\hfill \QED
%%%%%%%%%%%%%%%%%%%%%%%%%%%%%%%%%%%%%%%%%%%%%%%%%%%%%%%%%%%%%%%%
\\ \\
From Theorem \ref{SpiBW}, we can obtain the asymptotic behaviour of the exit times from a cone
for isotropic stable processes which generalizes a recent result in \cite{VaY11a}:
%%%%%%%%%%%%%%%%%%%%%%%%%%%%%%%%%%%%%%%%%%%%%%%%%%%%%%%%%%%%%%%%
\begin{prop}\label{prop1}
For $c\rightarrow \infty$, for every $x>0$, we have the weak convergence:
\beq\label{res}
\left(\frac{1}{c} \; \log\left(T^{\theta}_{x\sqrt{c}}\right), x\geq 0\right) \overset{{(d)}}{\underset{c\rightarrow \infty}\Longrightarrow}
\left(\tau^{(1/2)}_{\sqrt{1/r(\alpha)}} \ , x\geq 0\right),
\eeq
where for every $y>0$, $\tau^{(1/2)}_{y}$ stands for the $\frac{1}{2}$-stable process defined by: $\tau^{(1/2)}_{y}\equiv\inf\{ t:B_{t}=y \}$.
\end{prop}
%%%%%%%%%%%%%%%%%%%%%%%%%%%%%%%%%%%%%%%%%%%%%%%%%%%%%%%%%%%%%%%%
%%%%%%%%%%%%%%%%%%%%%%%%%%%%%%%%%%%%%%%%%%%%%%%%%%%%%%%%%%%%%%%%
\noindent{\textbf{Proof of Proposition \ref{prop1}:}} \\
We rely now upon Theorem \ref{SpiBW}, the analogue of Spitzer's Theorem for stable processes by Bertoin and Werner:
\beq
\Theta^{(c)}_{t}\equiv\left(c^{-1/2}\theta_{\exp(ct)},t\geq0\right)\overset{{(d)}}{\underset{c\rightarrow \infty}\Longrightarrow}
\left(B_{r(\alpha)t},t\geq0\right).
\eeq
Hence, for every $x>0$,
\beq\label{cvg}
\frac{1}{c} \ \log\left(T^{\theta}_{x\sqrt{c}}\right)&=& \frac{1}{c} \log \left(\inf\left\{t: \theta_{t}>x\sqrt{c}\right\}\right) \nonumber \\
&\stackrel{t=\exp(cs)}{=}& \frac{1}{c} \log \left(\inf\left\{e^{cs}:\frac{1}{\sqrt{c}} \ \theta_{\exp(cs)}>x\right\}\right) \nonumber \\
&=& \inf\left\{s:\frac{1}{\sqrt{c}} \ \theta_{\exp(cs)}>x\right\} \nonumber \\
&\stackrel{c\rightarrow \infty}{\longrightarrow}& \inf\left\{s:B_{r(\alpha)s}>x\right\} \nonumber \\
&=& \inf\left\{s:\sqrt{r(\alpha)}B_{s}>x\right\}\equiv \tau^{(1/2)}_{x/\sqrt{r(\alpha)}} \ .
\eeq
Moreover, from Theorem 7.1 in \cite{Whi80}, we know that the first passage time function mapping is continuous,
thus, we deduce (\ref{res}).
\hfill \QED
%%%%%%%%%%%%%%%%%%%%%%%%%%%%%%%%%%%%%%%%%%%%%%%%%%%%%%%%%%%%%%%%
\\ \\
If we replace $c$ by $ac$, we can obtain several variants of Proposition \ref{prop1}
for the random times $T^{\theta}_{-bc,ac}$, $0<a,b\leq \infty$, for $c\rightarrow \infty$,
and $a,b>0$ fixed:
%%%%%%%%%%%%%%%%%%%%%%%%%%%%%%%%%%%%%%%%%%%%%%%%%%%%%%%%%%%%%%%%
\begin{cor}\label{coro1}
The following asymptotic results hold:
\beq
\frac{1}{c} \; \log\left(T^{\theta}_{\sqrt{ac}}\right) &\overset{{(law)}}{\underset{c\rightarrow \infty}\longrightarrow}& \tau^{(1/2)}_{\sqrt{a/r(\alpha)}} \ , \label{cvgvar1} \\
\frac{1}{c} \; \log\left(T^{|\theta|}_{\sqrt{ac}}\right) &\overset{{(law)}}{\underset{c\rightarrow \infty}\longrightarrow}& \tau^{|B|}_{\sqrt{a/r(\alpha)}} \ , \label{cvgvar2} \\
\frac{1}{c} \; \log\left(T^{\theta}_{-\sqrt{bc},\sqrt{ac}}\right) &\overset{{(law)}}{\underset{c\rightarrow \infty}\longrightarrow}& \tau^{B}_{-\sqrt{b/r(\alpha)},\sqrt{a/r(\alpha)}} \ , \label{cvgvar3}
\eeq
where for every $x,y>0$, $\tau^{|B|}_{x}\equiv\inf\{ t:|B_{t}|=x \}$ and
$\tau^{B}_{-y,x}\equiv\inf\{ t:B_{t}\notin(-y,x) \}$.
\end{cor}
%%%%%%%%%%%%%%%%%%%%%%%%%%%%%%%%%%%%%%%%%%%%%%%%%%%%%%%%%%%%%%%%
%%%%%%%%%%%%%%%%%%%%%%%%%%%%%%%%%%%%%%%%%%%%%%%%%%%%%%%%%%%%%%%%
\begin{prop}\label{propTtheta}
The following asymptotic result for $\alpha\in(0,2)$, holds: for every $b>0$,
\beq\label{resTtheta}
P\left(T^{\theta}_{b\sqrt{\log t}}>t\right)\stackrel{t\rightarrow\infty}{\longrightarrow}\mathrm{erf}\left(\frac{b}{\sqrt{2r(\alpha)}}\right) \ ,
\eeq
where $\mathrm{erf}(x)\equiv \frac{2}{\sqrt{\pi}}\int^{x}_{0} e^{-y^{2}} dy$ is the error function.
\end{prop}
%%%%%%%%%%%%%%%%%%%%%%%%%%%%%%%%%%%%%%%%%%%%%%%%%%%%%%%%%%%%%%%%
%%%%%%%%%%%%%%%%%%%%%%%%%%%%%%%%%%%%%%%%%%%%%%%%%%%%%%%%%%%%%%%%
\noindent{\textbf{Proof of Proposition \ref{propTtheta}:}} \\
Using the notation of Theorem \ref{SpiBW}, for every $b>0$, we have:
\beqq
P\left(T^{\theta}_{b\sqrt{\log t}}>t\right)&=& P\left(\sup_{u\leq t} \theta_{u}<b\sqrt{\log t}\right)\stackrel{u=t^{v}}{=}
P\left(\sup_{v\leq 1 } (\log t)^{-1/2} \theta_{t^{v}}<b \right) \\
&\stackrel{t=e^{c}}{=}& P\left(\sup_{v\leq 1 }  c^{-1/2} \theta_{\exp(cv)} <b \right)
\eeqq
Hence, using Theorem \ref{SpiBW} for $t\rightarrow\infty$,
we deduce:
\beqq
P\left(T^{\theta}_{b\sqrt{\log t}}>t\right) &\stackrel{t\rightarrow\infty}{\longrightarrow}&
P\left(\sup_{v\leq 1} \sqrt{r(\alpha)}B_{v}<b\right)=P\left(|B_{1}|<\frac{b}{\sqrt{r(\alpha)}}\right) \\
&=& 2\int^{b/\sqrt{r(\alpha)}}_{0} \frac{dw}{\sqrt{2\pi}} \ e^{-w^{2}/2} \ ,
\eeqq
and changing the variables $w=y\sqrt{2}$, we obtain (\ref{resTtheta}).
\hfill \QED
%%%%%%%%%%%%%%%%%%%%%%%%%%%%%%%%%%%%%%%%%%%%%%%%%%%%%%%%%%%%%%%%
\\ \\
As mentioned in \cite{BeW96}, because an isotropic stable L\'{e}vy process $Z$
is transient, the difference between $\theta$ and the winding number around an arbitrary fixed $z\neq1$
is bounded and converges as $t\rightarrow \infty$.
Hence, with $\left(\theta^{i}_{t},t>0\right)$, $1\leq i\leq n$ denoting the continuous total angle wound of
$Z$ of index $\alpha\in(0,2)$ around $z^{i}$ ($z^{1},\ldots,z^{n}$ are $n$ distinct points in the complex plane $\mathbb{C}$) up to time $t$,
we obtain the following concerning the finite dimensional distributions
(windings around several points):
%%%%%%%%%%%%%%%%%%%%%%%%%%%%%%%%%%%%%%%%%%%%%%%%%%%%%%%%%%%%%%%%
\begin{prop}\label{finidim}
For isotropic stable L\'{e}vy processes of index $\alpha\in(0,2)$, we have:
\beq
\left(\frac{\theta^{i}_{t}}{\sqrt{\log t}},1\leq i\leq n\right) \overset{{(d)}}{\underset{t\rightarrow \infty}\Longrightarrow}
\left(\sqrt{r(\alpha)}B^{i}_{1},1\leq i\leq n\right),
\eeq
where $\left(B^{i}_{s},1\leq i\leq n,s\geq0\right)$ is an $n$-dimensional Brownian motion and
$r(\alpha)$ is given by (\ref{cst}).
\end{prop}
%%%%%%%%%%%%%%%%%%%%%%%%%%%%%%%%%%%%%%%%%%%%%%%%%%%%%%%%%%%%%%%%

%%%%%%%%%%%%%%%%%%%%%%%%%%%%%%%%%%%%%%%%%%%%%%%%%%%%%%%%%%%%%%%%
\section{Small time asymptotics}\label{smallt}
%%%%%%%%%%%%%%%%%%%%%%%%%%%%%%%%%%%%%%%%%%%%%%%%%%%%%%%%%%%%%%%%
We turn now our study to the behaviour of $\theta_{t}$ for $t\rightarrow0$.
%%%%%%%%%%%%%%%%%%%%%%%%%%%%%%%%%%%%%%%%%%%%%%%%%%%%%%%%%%%%%%%%
\begin{theo}\label{DV}
For $\alpha\in(0,2)$, the following convergence in law holds:
\beq\label{maintsmall}
\left(t^{-1/\alpha}\rho_{ts},s\geq0\right)\overset{{(d)}}{\underset{t\rightarrow0}\Longrightarrow} \left(\zeta_{s}, s\geq0\right) \ ,
\eeq
where $\left(\zeta_{s}, s\geq0\right)$ is a symmetric 1-dimensional $\alpha$-stable process and
the convergence in distribution is considered on $D(\left[\right.0,\infty\left.\right),\mathbb{R})$
endowed with the Skorohod topology.
\end{theo}
%%%%%%%%%%%%%%%%%%%%%%%%%%%%%%%%%%%%%%%%%%%%%%%%%%%%%%%%%%%%%%%%
%%%%%%%%%%%%%%%%%%%%%%%%%%%%%%%%%%%%%%%%%%%%%%%%%%%%%%%%%%%%%%%%
\noindent{\textbf{Proof of Theorem \ref{DV}:}} \\
From Lemma \ref{lemma2}, we use the L\'{e}vy measure, say $\tilde{\pi}$, of $\theta_{A(\cdot)}$ (thus the L\'{e}vy measure of $\rho$)
and we prove that for $t\rightarrow0$ it converges to the L\'{e}vy measure of a 1-dimensional $\alpha$-stable process.
Indeed, with $$L\equiv\frac{\alpha \ 2^{-1+\alpha/2} \Gamma(1+\alpha/2)}{\pi\Gamma(1-\alpha/2)},$$
and $z$ denoting a number in $\mathbb{C}$, using polar coordinates, we have:
\beqq
\phi\left(1+z\right)=\int_{\mathbb{C}} \frac{dz}{1+z}=
2L \int^{\pi}_{0}\int^{\infty}_{0} \frac{r \ dr \ d\varphi}{\left(1+r^{2}-2r\cos\varphi\right)^{1+\alpha/2}} \ .
\eeqq
We remark that: $$1+r^{2}-2r\cos\varphi=\left(r-\cos \varphi\right)^{2}+\sin^{2}\varphi \ ,$$
hence, changing the variables $\left(r-\cos \varphi\right)^{2}= t^{-1} \sin^{2}\varphi$
and denoting by: $$B(y;a,b)=\int^{y}_{0}u^{a-1}(1-u)^{b-1} du \ ,$$ the incomplete Beta function,
for $\varphi>0$ (we can repeat the same arguments for $\varphi<0$) we have:
\beq
\tilde{\pi}(d\varphi)&=& d\varphi \ 2L \int^{\infty}_{0} \frac{r \ dr }{\left(\left(r-\cos \varphi\right)^{2}+\sin^{2}\varphi\right)^{1+\alpha/2}} \nonumber \\
&=& d\varphi \ \frac{2L}{2} \left(\frac{2}{\alpha}+\cos\varphi \left(1-\cos^{2}\varphi\right)^{-\frac{1}{2}-\frac{\alpha}{2}}
\int^{1-\frac{1}{\cos^{2}\varphi}}_{0} t^{-\frac{1}{2}+\frac{\alpha}{2}}\left(1+t\right)^{-\frac{\alpha}{2}-1} dt\right) \nonumber \\
&\stackrel{u=-t}{=}& d\varphi \ L \left(\frac{2}{\alpha}+\cos\varphi \left(-1+\cos^{2}\varphi\right)^{-\frac{1}{2}-\frac{\alpha}{2}}
\int^{1-\frac{1}{\cos^{2}\varphi}}_{0} u^{-\frac{1}{2}+\frac{\alpha}{2}}\left(1-u\right)^{-\frac{\alpha}{2}-1} du\right) \nonumber \\
&=& d\varphi \ L \left(\frac{2}{\alpha}+ \cos\varphi \left(-1+\cos^{2}\varphi\right)^{-\frac{1}{2}-\frac{\alpha}{2}} B\left(1-\frac{1}{\cos^{2}\varphi};\frac{1}{2}+\frac{\alpha}{2},-\frac{\alpha}{2}\right)\right) \nonumber \\
&\stackrel{\varphi\sim0}{\thicksim}& \tilde{L}\varphi^{-1-\alpha} \ d\varphi \ , \label{varphismall}
\eeq
which is the L\'{e}vy measure of an $\alpha$-stable process.
The result now follows by standard arguments.
\hfill \QED
%%%%%%%%%%%%%%%%%%%%%%%%%%%%%%%%%%%%%%%%%%%%%%%%%%%%%%%%%%%%%%%%%
\\ \\
Concerning the clock $H$ and its increments, we have:
%%%%%%%%%%%%%%%%%%%%%%%%%%%%%%%%%%%%%%%%%%%%%%%%%%%%%%%%%%%%%%%%
\begin{theo}\label{theoclock}
The following a.s. convergence holds:
\beq\label{clocktsmall}
\left(\frac{H(ux)}{u}, x\geq 0\right) \overset{{a.s.}}{\underset{u\rightarrow0}\longrightarrow} \left(x, x\geq0\right) \ .
\eeq
\end{theo}
%%%%%%%%%%%%%%%%%%%%%%%%%%%%%%%%%%%%%%%%%%%%%%%%%%%%%%%%%%%%%%%%
%%%%%%%%%%%%%%%%%%%%%%%%%%%%%%%%%%%%%%%%%%%%%%%%%%%%%%%%%%%%%%%%%
\noindent{\textbf{Proof of Theorem \ref{theoclock}:}} \\
From the definition of the clock $H$ we have:
\beqq
\frac{H(ux)}{u}=\frac{1}{u} \int^{ux}_{0} \frac{ds}{|Z_{s}|^{\alpha}} \ .
\eeqq
Hence, for every $x_{0}>0$, we have:
\beq\label{prsmall1}
\sup_{x\leq x_{0}} \left|\frac{H(ux)-ux}{u}\right|&=&
\sup_{x\leq x_{0}} \frac{1}{u}\left|\int^{ux}_{0} \left(\frac{1}{|Z_{s}|^{\alpha}}-1\right) ds \ \right|
\leq \frac{1}{u}\int^{ux_{0}}_{0} \left|\frac{1}{|Z_{s}|^{\alpha}}-1\right| ds \nonumber \\
&\stackrel{s=uw}{=}& \int^{x_{0}}_{0} \left|\frac{1}{|Z_{uw}|^{\alpha}}-1\right| dw \
\overset{{a.s.}}{\underset{u\rightarrow0}\longrightarrow} 0 \ .
\eeq
because:
\beq
|Z_{u}|^{\alpha}\overset{{a.s.}}{\underset{u\rightarrow0}\longrightarrow} 1 \ .
\eeq
Thus, as (\ref{prsmall1}) is true for every $x_{0}>0$, we obtain (\ref{clocktsmall}).
\hfill \QED
%%%%%%%%%%%%%%%%%%%%%%%%%%%%%%%%%%%%%%%%%%%%%%%%%%%%%%%%%%%%%%%%
%%%%%%%%%%%%%%%%%%%%%%%%%%%%%%%%%%%%%%%%%%%%%%%%%%%%%%%%%%%%%%%%
\begin{rem}\label{remclocklarge}
We remark that this behaviour of the clock is different for the case $t\rightarrow\infty$,
where (\ref{prlarge1}) can be equivalently stated as:
\beq\label{clocktlargebis}
\left(\frac{H(ux)}{\log u}, x\geq 0\right) \overset{{(d)}}{\underset{u\rightarrow\infty}\Longrightarrow} \left(2^{-\alpha} \frac{\Gamma(1-\alpha/2)}{\Gamma(1+\alpha/2)} \ x, x\geq0\right) .
\eeq
\end{rem}
%%%%%%%%%%%%%%%%%%%%%%%%%%%%%%%%%%%%%%%%%%%%%%%%%%%%%%%%%%%%%%%%
\noindent{Using Theorems \ref{DV} and \ref{theoclock}, we obtain:}
%%%%%%%%%%%%%%%%%%%%%%%%%%%%%%%%%%%%%%%%%%%%%%%%%%%%%%%%%%%%%%%%
\begin{theo}\label{DV2}
With $\alpha\in(0,2)$, the family of processes $$\left(c^{-1/\alpha}\theta_{ct},t\geq0\right)$$
converges in distribution on $D(\left[\right.0,\infty\left.\right),\mathbb{R})$
endowed with the Skorohod topology, as $c\rightarrow0$, to a symmetric 1-dimensional $\alpha$-stable process
$\left(\zeta_{t},t\geq0\right)$.
\end{theo}
%%%%%%%%%%%%%%%%%%%%%%%%%%%%%%%%%%%%%%%%%%%%%%%%%%%%%%%%%%%%%%%%
%%%%%%%%%%%%%%%%%%%%%%%%%%%%%%%%%%%%%%%%%%%%%%%%%%%%%%%%%%%%%%%%
\noindent{\textbf{Proof of Theorem \ref{DV2}:}} \\
We will use Theorems \ref{DV} and \ref{theoclock}. More precisely,
we shall rely again upon the continuity of the composition function as studied in Theorem 3.1 in \cite{Whi80}. \\ \\
$\left.\mathrm{i}\right)$ First, concerning the clock $H$, for every $t>0$,
we have the almost sure convergence (\ref{clocktsmall}), which yields the weak convergence
(on $D(\left[\right.0,\infty\left.\right),\mathbb{R})$ endowed with the Skorohod topology)
of the family of processes $\tilde{H}^{(u)}(t)\equiv(u^{-1}H(ut),t\geq0)$
to $(t, t\geq0)$ as $u\rightarrow0$. \\ \\
$\left.\mathrm{ii}\right)$ We use another result of Whitt \cite{Whi80} which states that the composition function on
$D(\left[\right.0,\infty\left.\right),\mathbb{R})\times D(\left[\right.0,\infty\left.\right),\left[\right.0,\infty\left.\right))$
is continuous at each $(\rho,H)\in(D(\left[\right.0,\infty\left.\right),\mathbb{R}) \times C_{0}(\left[\right.0,\infty\left.\right),\left[\right.0,\infty\left.\right)))$,
with $D$ denoting the set of cadlag functions and $C_{0}$ the subset of strictly-increasing functions in $C$.
Hence, from Theorem \ref{DV} and (\ref{clocktsmall}), using the weak convergence of $\tilde{H}^{(c)}(\cdot)$ and of
$c^{-1/\alpha}\rho_{c}$, as $c\rightarrow 0$, we deduce: for every $t>0$,
\beq\label{cvgpr}
\frac{\theta_{ct}}{c^{1/\alpha}}=\frac{\rho_{H(ct)}}{c^{1/\alpha}}=\frac{\rho_{c(H(ct)/c)}}{c^{1/\alpha}} \overset{{(d)}}{\underset{c\rightarrow0}\Longrightarrow}\zeta_{t} \ ,
\eeq
where the convergence in distribution is viewed on $D(\left[\right.0,\infty\left.\right),\mathbb{R})$
endowed with the Skorohod topology.
\hfill \QED
%%%%%%%%%%%%%%%%%%%%%%%%%%%%%%%%%%%%%%%%%%%%%%%%%%%%%%%%%%%%%%%%%
\\ \\
From the previous results, we deduce the asymptotic behaviour, for $c\rightarrow0$, of the first exit times
from a cone for isotropic stable processes of index $\alpha\in(0,2)$ taking values in the complex plane:
%%%%%%%%%%%%%%%%%%%%%%%%%%%%%%%%%%%%%%%%%%%%%%%%%%%%%%%%%%%%%%%%
\begin{prop}\label{proptsmall}
For $c\rightarrow 0$, we have the weak convergence:
\beq\label{restsmall}
\left(\frac{1}{c} \ T^{\theta}_{c^{1/\alpha}x} \ , \ x\geq 0 \right) \overset{{(d)}}{\underset{c\rightarrow 0}\Longrightarrow}
\left(T^{\zeta}_{x} \ x\geq 0 \right) ,
\eeq
where for every $x$, $T^{\zeta}_{x}$ is the first hitting time defined by: $T^{\zeta}_{x}\equiv\inf\{ t:\zeta_{t}=x \}$.
\end{prop}
%%%%%%%%%%%%%%%%%%%%%%%%%%%%%%%%%%%%%%%%%%%%%%%%%%%%%%%%%%%%%%%%
%%%%%%%%%%%%%%%%%%%%%%%%%%%%%%%%%%%%%%%%%%%%%%%%%%%%%%%%%%%%%%%%
\noindent{\textbf{Proof of Proposition \ref{proptsmall}:}} \\
Using Theorem \ref{DV2}, we have:
\beqq
\frac{1}{c} \ T^{\theta}_{c^{1/\alpha}x}&=& \frac{1}{c} \inf\left\{t: \theta_{t}>c^{1/\alpha}x\right\}
\stackrel{t=cs}{=}\frac{1}{c} \inf\left\{cs: c^{-1/\alpha}\theta_{cs}>x\right\}
= \inf\left\{s: c^{-1/\alpha}\theta_{cs}>x\right\}\\
&{\overset{c\rightarrow 0}\longrightarrow}& \inf\left\{s: \zeta_{s}>x\right\} \ ,
\eeqq
which, using again the continuity of the first passage time function mapping (see Theorem 7.1 in \cite{Whi80}),
yields (\ref{restsmall}). \hfill \QED
%%%%%%%%%%%%%%%%%%%%%%%%%%%%%%%%%%%%%%%%%%%%%%%%%%%%%%%%%%%%%%%
\\ \\
Finally, we can obtain several variants of Proposition \ref{proptsmall} for the random times
$T^{\theta}_{-bc,ac}$, $0<a,b\leq \infty$ fixed, for $c\rightarrow 0$:
%%%%%%%%%%%%%%%%%%%%%%%%%%%%%%%%%%%%%%%%%%%%%%%%%%%%%%%%%%%%%%%%
\begin{cor}\label{coro1small}
The following asymptotic results hold:
\beq
\frac{1}{c} \; T^{\theta}_{ac^{1/\alpha}} &\overset{{(law)}}{\underset{c\rightarrow 0}\longrightarrow}&
T^{\zeta}_{a} \ , \label{cvgvar1small} \\
\frac{1}{c} \; T^{|\theta|}_{ac^{1/\alpha}} &\overset{{(law)}}{\underset{c\rightarrow 0}\longrightarrow}&
T^{|\zeta|}_{a} \ , \label{cvgvar2small} \\
\frac{1}{c} \; T^{\theta}_{-bc^{1/\alpha},ac^{1/\alpha}} &\overset{{(law)}}{\underset{c\rightarrow 0}\longrightarrow}&
T^{\zeta}_{-b,a} \ , \label{cvgvar3small}
\eeq
where, for every $x,y>0$, $T^{|\zeta|}_{x}\equiv\inf\{ t:|\zeta_{t}|=x \}$ and
$T^{\zeta}_{-y,x}\equiv\inf\{ t:\zeta_{t}\notin(-y,x) \}$.
\end{cor}
%%%%%%%%%%%%%%%%%%%%%%%%%%%%%%%%%%%%%%%%%%%%%%%%%%%%%%%%%%%%%%%%
%%%%%%%%%%%%%%%%%%%%%%%%%%%%%%%%%%%%%%%%%%%%%%%%%%%%%%%%%%%%%%%%
\begin{rem}\label{remt1}
\emph{\textbf{(Windings of planar stable processes in $\left(\right.t,1\left.\right]$ for $t\rightarrow0$})} \\
$\left.i\right)$ We consider now our stable process $Z$ starting from 0 and we want to investigate its windings
in $\left(\right.t,1\left.\right]$ for $t\rightarrow0$. We know that it doesn't visit again the origin but
it winds a.s. infinitely often around it, hence, its winding process $\theta$ in $\left(\right.t,1\left.\right]$ is well-defined.
With obvious notation, concerning now the clock $H_{\left(\right.t,1\left.\right]}=\int^{1}_{t} du \ |Z_{u}|^{-\alpha}$,
the change of variables $u=tv$ and the stability property, i.e.: $Z_{tv}\stackrel{(law)}{=}t^{1/\alpha}Z_{v}$ \ , yield:
\beqq
H_{\left(\right.t,1\left.\right]}=\int^{1/t}_{1} \frac{t \ dv}{|Z_{tv}|^{\alpha}}\stackrel{(law)}{=} \int^{1/t}_{1} \frac{dv}{|Z_{v}|^{\alpha}}=H_{\left(\right.1,1/t\left.\right]} \ .
\eeqq
Hence, as before, using Whitt's result \cite{Whi80} on the continuity of the composition function on
$D(\left[\right.0,\infty\left.\right),\mathbb{R})\times D(\left[\right.0,\infty\left.\right),\left[\right.0,\infty\left.\right))$
at each $(\rho,H)\in(D(\left[\right.0,\infty\left.\right),\mathbb{R}) \times C_{0}(\left[\right.0,\infty\left.\right),\left[\right.0,\infty\left.\right)))$,
we have (with obvious notation):
\beq
\theta_{\left(\right.t,1\left.\right]}=\rho_{H_{\left(\right.t,1\left.\right]}}
\stackrel{(law)}{=}\rho_{H_{\left(\right.1,1/t\left.\right]}}
=\theta_{\left(\right.1,1/t\left.\right]} \ .
\eeq
\\
The only difference with respect to the "normal" stable case is that the winding process is considered from 1 and not from 0,
but this doesn't provoke any problem. \\
Hence, Bertoin and Werner's Theorem \ref{SpiBW} is still valid for $Z$ in $\left(\right.t,1\left.\right]$,
$t\rightarrow0$: for $\alpha\in(0,2)$:
\beq
\frac{1}{\sqrt{\log(1/t)}} \ \theta_{\left(\right.t,1\left.\right]}
\stackrel{(law)}{=}\frac{1}{\sqrt{\log(1/t)}} \ \theta_{\left(\right.1,1/t\left.\right]}
\overset{{(d)}}{\underset{t\rightarrow 0}\Longrightarrow}\sqrt{r(\alpha)}N,
\eeq
with $r(\alpha)$ defined in (\ref{cst}) and $N\backsim\mathcal{N}(0,1)$.
\\ \\
$\left.ii\right)$ We note that this study is also valid for a planar Brownian motion starting from 0
in $\left(\right.t,1\left.\right]$ for $t\rightarrow0$ and for planar stable processes or
planar Brownian motion starting both from a point different from 0 (in order to have an well-defined winding number)
in $\left[0,1\right]$. In particular, for planar Brownian motion $\mathcal{Z}$ with associated winding number $\vartheta$, we obtain that Spitzer's law is still valid for $t\rightarrow0$
(see e.g. \cite{ReY99,LeG92}):
\beq
\frac{2}{\log(1/t)} \ \vartheta_{\left(\right.t,1\left.\right]}
\overset{{(law)}}{\underset{t\rightarrow 0}\longrightarrow} C_{1},
\eeq
where $C_{1}$ is a standard Cauchy variable. We also remark that this result could also be obtained
from a time inversion argument, that is: with $\mathcal{Z}'$ denoting another planar Brownian motion starting from 0,
with winding number $\vartheta'$,
by time inversion we have: $\mathcal{Z}_{u}=u\mathcal{Z}'_{1/u}$. Changing now the variables $u=1/v$,  we obtain:
\beqq
\vartheta_{\left(\right.t,1\left.\right]}&\equiv&\mathrm{Im} \left(\int^{1}_{t} \frac{d\mathcal{Z}_{u}}{\mathcal{Z}_{u}}\right)=\mathrm{Im} \left(\int^{1}_{t} \frac{d(u\mathcal{Z}'_{1/u})}{u\mathcal{Z}'_{1/u}}\right)=
\mathrm{Im} \left(\int^{1}_{t} \frac{d\mathcal{Z}'_{1/u}}{\mathcal{Z}'_{1/u}}\right) \\
&=& \mathrm{Im} \left(\int^{1/t}_{1} \frac{d\mathcal{Z}'_{v}}{\mathcal{Z}'_{v}}\right)\equiv\vartheta'_{\left(\right.1,1/t\left.\right]} \ ,
\eeqq
and we continue as before. \\
Note that this time inversion argument is NOT valid for planar stable processes.
\end{rem}
%%%%%%%%%%%%%%%%%%%%%%%%%%%%%%%%%%%%%%%%%%%%%%%%%%%%%%%%%%%%%%%%

%%%%%%%%%%%%%%%%%%%%%%%%%%%%%%%%%%%%%%%%%%%%%%%%%%%%%%%%%%%%%%%%
\section{The Law of the Iterated Logarithm (LIL)}\label{secLIL}
%%%%%%%%%%%%%%%%%%%%%%%%%%%%%%%%%%%%%%%%%%%%%%%%%%%%%%%%%%%%%%%%
In this Section, we shall use some notation introduced in \cite{DoM02a,DoM02b}. Recall (\ref{lm}); then, for all $x>0$,
because $\rho$ is symmetric, we define:
\beqq
L(x)=2\hat{\mu}(x)=2\mu(x,+\infty), \ \ U(x)=2\int^{x}_{0}yL(y)dy \ .
\eeqq
We remark that $U$ plays essentially the role of the truncated variance in the random walk case (see e.g. \cite{DoM00}). \\
Hence, from (\ref{maintsmall}), for $t\rightarrow0$, we have ($\tilde{K}(\alpha)$ is a constant depending on $\alpha$):
\beq\label{U}
U(x)\stackrel{x\sim0}{\thicksim}\tilde{K}(\alpha)x^{2-\alpha}.
\eeq
Then, we obtain the following Law of the Iterated Logarithm (LIL) for L\'{e}vy processes for small times:
%%%%%%%%%%%%%%%%%%%%%%%%%%%%%%%%%%%%%%%%%%%%%%%%%%%%%%%%%%%%%%%%
\begin{theo}\emph{\textbf{(LIL for L\'{e}vy processes for small times)}}\label{DV4} \\
For any non-decreasing function $f>0$,
\beq\label{LIL2}
\limsup_{t\rightarrow0} \frac{\rho_{t}}{f(t)}=
\left\{
  \begin{array}{ll}
    0 \ ; \\
    \infty
  \end{array}
\right.
\ a.s. \
\Leftrightarrow
\int^{\infty}_{1}(f(t))^{-\alpha}dt
\left\{
  \begin{array}{ll}
    <\infty \ ; \\
    =\infty \ .
  \end{array}
\right.
\eeq
\end{theo}
%%%%%%%%%%%%%%%%%%%%%%%%%%%%%%%%%%%%%%%%%%%%%%%%%%%%%%%%%%%%%%%%
We can reformulate Theorem \ref{DV4} by using the skew-product representation (\ref{skew-product2}) stating: $\rho_{t}=\theta_{A(t)}$ \ ,
in order to deduce a LIL for the winding process $\theta_{A(\cdot)}$ for small times.
%%%%%%%%%%%%%%%%%%%%%%%%%%%%%%%%%%%%%%%%%%%%%%%%%%%%%%%%%%%%%%%%
\begin{cor}\label{DV5}
For any non-decreasing function $f>0$,
\beq\label{LIL}
\limsup_{t\rightarrow0} \frac{\theta_{A(t)}}{f(t)}=
\left\{
  \begin{array}{ll}
    0 \ ; \\
    \infty
  \end{array}
\right.
\ a.s. \
\Leftrightarrow
\int^{\infty}_{1}(f(t))^{-\alpha}dt
\left\{
  \begin{array}{ll}
    <\infty \ ; \\
    =\infty \ .
  \end{array}
\right.
\eeq
\end{cor}
%%%%%%%%%%%%%%%%%%%%%%%%%%%%%%%%%%%%%%%%%%%%%%%%%%%%%%%%%%%%%%%%
%%%%%%%%%%%%%%%%%%%%%%%%%%%%%%%%%%%%%%%%%%%%%%%%%%%%%%%%%%%%%%%%
\noindent{\textbf{Proof of Theorem \ref{DV4}:}} \\
First, we define:
\beq
h(y)=y^{-2}U(y), \ \ y>0.
\eeq
Then, we consider $t_{n}=2^{-n}$ and we note that (Cauchy's test):
\beqq
I(f)\equiv \int^{\infty}_{1} dt \ h(f(t))<\infty \ \Longleftrightarrow \ \sum^{\infty}_{n=1} t_{n} h(f(t_{n}))<\infty \ .
\eeqq
Using now Lemma 2 from Doney and Maller \cite{DoM02a}, because $\rho$ is symmetric, there exists a positive constant $c_{1}$
such that for every $x>0$, $t>0$,
\beq\label{DM}
P\left(\rho_{t}\geq x\right)\leq P\left(\sup_{0\leq u\leq t}\rho_{u}\geq x\right)\leq c_{1}t \ h(x)= c_{1}t \ \frac{U(x)}{x^{2}} \ .
\eeq
Thus:
\beqq
\sum^{\infty}_{n=1} P\left(\rho_{t_{n-1}}\geq f(t_{n})\right)\leq c_{1} \sum^{\infty}_{n=1}t_{n}
\frac{U(f(t_{n}))}{(f(t_{n}))^{2}}.
\eeqq
From (\ref{U}), we have that,
$$\frac{U(f(t_{n}))}{(f(t_{n}))^{2}}\stackrel{t_{n}\sim0}{\thicksim}(f(t_{n}))^{-\alpha}.$$
Hence, when $I(f)<\infty$, from Borel-Cantelli Lemma we have that with probability 1,
$\rho_{t_{n-1}}\leq f(t_{n})$ for all $n$'s, except for a finite number of them. Now, from a monotonicity
argument  for $f$, if $t\in[t_{n},t_{n-1}]$, we have that: $\rho_{t_{n-1}}\leq f(t_{n})\leq f(t)$
for every $t$ sufficiently small. It follows now that $\lim_{t\rightarrow0}(\rho_{t}/f(t))\leq 1$ a.s.
Finally, we remark that as $I(f)<\infty$, we also have that $I(\varepsilon f)<\infty$, for arbitrarily
small $\varepsilon>0$ and follows that $\rho_{t}/f(t)\rightarrow0$ a.s. \\
The proof of the second statement follows from the same kind of arguments. Indeed,
using Lemma 2 from Doney and Maller \cite{DoM02a} and the fact that $\rho$ is symmetric, there exists a positive constant $c_{2}$
such that for every $x>0$, $t>0$,
\beq\label{DM2}
P\left(\sup_{0\leq u\leq t}\rho_{u}\leq x\right)\leq \frac{c_{2}}{t \ h(x)} \ .
\eeq
Hence:
\beqq
\sum^{\infty}_{n=1} P\left(\rho_{t_{n-1}}\leq f(t_{n-1}) \right)
&\leq& \sum^{\infty}_{n=1} P\left(\sup_{0\leq u\leq t_{n-1}} \rho_{u}\leq f(t_{n-1}) \right) \\
&\leq& \sum^{\infty}_{n=1}\frac{c_{2}}{t_{n-1} \ h(f(t_{n-1}))} \ .
\eeqq
Thus, for $I(f)=\infty$ (or equivalently $\sum h(f(t_{n-1}))=\infty$), Borel-Cantelli Lemma yields that for every $n$,
a.s. $\rho_{t_{n-1}}> f(t_{n-1})$ infinitely often, which finishes the proof.
\hfill \QED
%%%%%%%%%%%%%%%%%%%%%%%%%%%%%%%%%%%%%%%%%%%%%%%%%%%%%%%%%%%%%%%%
%%%%%%%%%%%%%%%%%%%%%%%%%%%%%%%%%%%%%%%%%%%%%%%%%%%%%%%%%%%%%%%%%%
\begin{rem}
For other kinds of LIL for L\'{e}vy processes for small times e.g. of the Chung type, see \cite{ADS10}
and the references therein.
\end{rem}
%%%%%%%%%%%%%%%%%%%%%%%%%%%%%%%%%%%%%%%%%%%%%%%%%%%%%%%%%%%%%%%%%%
%%%%%%%%%%%%%%%%%%%%%%%%%%%%%%%%%%%%%%%%%%%%%%%%%%%%%%%%%%%%%%%%
\begin{theo}\emph{\textbf{(LIL for the angular part of planar stable processes for small times)}}\label{DV6} \\
For any non-decreasing function $f>0$,
\beq\label{LIL3}
\limsup_{t\rightarrow0} \frac{\theta_{t}}{f(t)}=
\left\{
  \begin{array}{ll}
    0 \ ; \\
    \infty
  \end{array}
\right.
\ a.s. \
\Leftrightarrow
\int^{\infty}_{1}(f(t))^{-\alpha}dt
\left\{
  \begin{array}{ll}
    <\infty \ ; \\
    =\infty \ .
  \end{array}
\right.
\eeq
\end{theo}
%%%%%%%%%%%%%%%%%%%%%%%%%%%%%%%%%%%%%%%%%%%%%%%%%%%%%%%%%%%%%%%%
%%%%%%%%%%%%%%%%%%%%%%%%%%%%%%%%%%%%%%%%%%%%%%%%%%%%%%%%%%%%%%%%
\noindent{\textbf{Proof of Theorem \ref{DV6}:}} \\
We use the skew-product representation (\ref{skew-product2}) together with (\ref{clocktsmall}), which
essentially writes: $t^{-1}H(t)\overset{{a.s.}}{\underset{t\rightarrow0}\longrightarrow}1$.
Thus, for every $\varepsilon, \delta>0$, there exists $t_{0}>0$ such that:
\beq
P\left(\frac{H(t)}{t}\leq 1+\varepsilon\right)\geq 1-\delta, \ \ \mathrm{for} \ \ t\leq t_{0} \ .
\eeq
We define now the setting: $$\mathcal{K}\equiv\mathcal{K}(\omega)\equiv\left\{\omega:\frac{H(t)}{t}\leq 1+\varepsilon\right\} \ ,
\ \ \mathrm{thus:} \ \ \overline{\mathcal{K}}\equiv\overline{\mathcal{K}}(\omega)\equiv\left\{\omega:\frac{H(t)}{t}\geq 1+\varepsilon\right\} ,$$
hence, there exists $t_{0}>0$ such that: for every $t\leq t_{0}$,
$$P(\mathcal{K})\geq 1-\delta \ \ \mathrm{and} \ \ P(\overline{\mathcal{K}})\leq \delta \ .$$
Hence, choosing $\delta>0$ small enough, it suffices to restrict our study in the set $\mathcal{K}$
and it follows that:
\beqq
P\left(\sup_{0\leq u\leq t} \theta_{u}>x\right)&=& P\left(\sup_{0\leq u\leq t} \rho_{H(u)}>x\right)
= P\left(\left\{\sup_{0\leq u\leq t} \rho_{H(u)}>x \right\} \ \cap \ \mathcal{K} \right) \\
&\leq& P\left(\sup_{0\leq u\leq t} \rho_{u(1+\varepsilon)}>x\right)
\eeqq
Changing now the variables $\tilde{u}=u(1+\varepsilon)$, and invoking
(\ref{DM}), there exists another positive constant $c_{3}$ such that, for every $x>0$ and $t>0$:
\beq\label{proofDV5}
P\left(\sup_{0\leq u\leq t} \theta_{u}>x\right) \leq
P\left(\sup_{0\leq \tilde{u}\leq t(1+\varepsilon)} \rho_{\tilde{u}}>x\right)
\leq c_{3} \ t(1+\varepsilon) \ \frac{U(x)}{x^{2}} \ .
\eeq
Mimicking now the proof of Theorem \ref{DV4}, we obtain the first statement. \\
For the second statement, we use the settings:
$$\mathcal{K}'\equiv\mathcal{K}'(\omega)\equiv\left\{\omega:\frac{H(t)}{t}\geq 1-\varepsilon\right\} \ ,
\ \ \mathrm{thus:} \ \ \overline{\mathcal{K}}'\equiv\overline{\mathcal{K}}'(\omega)\equiv\left\{\omega:\frac{H(t)}{t}\leq 1-\varepsilon\right\} .$$
Hence, for every $\varepsilon, \delta>0$, there exists $t_{0}>0$ such that: for every $t\leq t_{0}$,
$$P(\mathcal{K}')\geq 1-\delta \ \ \mathrm{and} \ \ P(\overline{\mathcal{K}}')\leq \delta \ .$$
As before, we choose $\delta>0$ small enough and we restrict our study in the set $\mathcal{K}'$.
The proof finishes by repeating the arguments of the proof of Theorem \ref{DV4}.
\hfill \QED
%%%%%%%%%%%%%%%%%%%%%%%%%%%%%%%%%%%%%%%%%%%%%%%%%%%%%%%%%%%%%%%%

%%%%%%%%%%%%%%%%%%%%%%%%%%%%%%%%%%%%%%%%%%%%%%%%%%%%%%%%%%%%%%%%
\section{The planar Brownian motion case}\label{secBM}
%%%%%%%%%%%%%%%%%%%%%%%%%%%%%%%%%%%%%%%%%%%%%%%%%%%%%%%%%%%%%%%%
Before starting, we remark that the notations used in this Section are independent from
the ones used in the text up to now. \\ \\
In this Section, we state and give a new proof of the analogue of Theorem \ref{DV2}
for the planar Brownian motion case, which is equivalent to a result obtained in \cite{VaY11a}.
For this purpose, and in order to avoid complexity, we will use the same notation as in the "stable" case.
Hence, for a planar BM $\mathcal{Z}$ starting from a point different $z_{0}$ from 0 (without loss of generality, let $z_{0}=1$)
and with $\vartheta=\left(\vartheta_{t},t\geq0\right)$ denoting now the (well defined - see eg. \cite{ItMK65}) continuous winding process,
we have the skew product representation (see e.g. \cite{ReY99}):
\beq\label{skew-productBM}
\log\left|\mathcal{Z}_{t}\right|+i\vartheta_{t}\equiv\int^{t}_{0}\frac{d\mathcal{Z}_{s}}{\mathcal{Z}_{s}}=\left(
\beta_{u}+i\gamma_{u}\right)
\Bigm|_{u=\mathcal{H}_{t}=\int^{t}_{0}\frac{ds}{\left|\mathcal{Z}_{s}\right|^{2}}} \ ,
\eeq
where $(\beta_{u}+i\gamma_{u},u\geq0)$ is another planar Brownian motion starting from $\log 1+i0=0$.
The Bessel clock $\mathcal{H}$ plays a key role in many aspects of the study of the winding number process
$(\vartheta_{t},t\geq0)$ (see e.g. \cite{Yor80}). We shall also make use of the inverse of $\mathcal{H}$, which is given by:
\beq
\mathcal{H}^{-1}_{u}=\inf\{t\geq0:\mathcal{H}(t)>u\}=\int^{u}_{0}ds \; \exp(2\beta_{s})=\mathcal{A}_{u} \ .
\eeq
Rewriting (\ref{skew-productBM}) as:
\beq\label{skew-productBM2}
\log\left|\mathcal{Z}_{t}\right|=\beta_{\mathcal{H}_{t}}; \ \ \vartheta_{t}=\gamma_{\mathcal{H}_{t}} ,
\eeq
we easily obtain that the two $\sigma$-fields $\sigma \{\left|\mathcal{Z}_{t}\right|,t\geq0\}$ and $\sigma \{\beta_{u},u\geq0\}$ are identical, whereas $(\gamma_{u},u\geq0)$ is independent from $(\left|\mathcal{Z}_{t}\right|,t\geq0)$, a fact that is in contrast to what happens in the "stable" case.
%%%%%%%%%%%%%%%%%%%%%%%%%%%%%%%%%%%%%%%%%%%%%%%%%%%%%%%%%%%%%%%%
\begin{theo}\label{DV3}
The family of processes $$\left(c^{-1/2}\vartheta_{ct},t\geq0\right)$$
converges in distribution, as $c\rightarrow0$, to a 1-dimensional Brownian motion $\left(\gamma_{t},t\geq0\right)$.
\end{theo}
%%%%%%%%%%%%%%%%%%%%%%%%%%%%%%%%%%%%%%%%%%%%%%%%%%%%%%%%%%%%%%%%
%%%%%%%%%%%%%%%%%%%%%%%%%%%%%%%%%%%%%%%%%%%%%%%%%%%%%%%%%%%%%%%%
\begin{proof}
We split the proof in two parts: \\ \\
$\left.\mathrm{i}\right)$ First, repeating the arguments in the proof of Theorem \ref{theoclock}
with $\alpha=2$, we obtain:
\beq\label{clockBM}
\left(\frac{\mathcal{H}(ux)}{u}, x\geq 0\right) \overset{{a.s.}}{\underset{u\rightarrow0}\longrightarrow} \left(x, x\geq0\right) .
\eeq
which also implies the weak convergence:
\beq\label{clockBM2}
\left(\frac{\mathcal{H}(ux)}{u}, x\geq 0\right) \overset{{(d)}}{\underset{u\rightarrow0}\Longrightarrow} \left(x, x\geq0\right) .
\eeq
\\
$\left.\mathrm{ii}\right)$ Using the skew product representation (\ref{skew-productBM2})
and the scaling property of BM together with (\ref{clockBM}), we have that for every $s>0$:
\beq
t^{-1/2} \vartheta_{st}=t^{-1/2} \gamma_{\mathcal{H}(st)}\stackrel{(law)}{=}\sqrt{\frac{\mathcal{H}(st)}{t}} \gamma_{1}
\overset{{a.s.}}{\underset{t\rightarrow0}\longrightarrow}\sqrt{s}\gamma_{1}\stackrel{(law)}{=}\gamma_{s} \ ,
\eeq
which finishes the proof. \\
We remark that for part (ii) of the proof, we could also invoke Whitt's Theorem 3.1 concerning the composition function
\cite{Whi80}, however, the independence in the planar Brownian motion case simplifies the proof.
\qed
\end{proof}
%%%%%%%%%%%%%%%%%%%%%%%%%%%%%%%%%%%%%%%%%%%%%%%%%%%%%%%%%%%%%%%
From Theorem \ref{DV3} now, with $T^{|\vartheta|}_{c}\equiv\inf\{t:|\vartheta_{t}|=c\}$ and
$T^{|\gamma|}_{c}\equiv\inf\{t:|\gamma_{t}|=c\}$, $(c>0)$,
we deduce for the exit time from a cone of planar BM (this result has already been obtained in \cite{VaY11a},
where one can also find several variants):
%%%%%%%%%%%%%%%%%%%%%%%%%%%%%%%%%%%%%%%%%%%%%%%%%%%%%%%%%%%%%%%%%%
\begin{cor} \label{corT}
The following convergence in law holds:
\beq
\left(\frac{1}{c^{2}} \; T^{|\vartheta|}_{cx}, x\geq0\right) \overset{{(law)}}{\underset{c\rightarrow 0}\longrightarrow}
\left(T^{|\gamma|}_{x}, \ x\geq0\right).
\eeq
\end{cor}
%%%%%%%%%%%%%%%%%%%%%%%%%%%%%%%%%%%%%%%%%%%%%%%%%%%%%%%%%%%%%%%%%%
%%%%%%%%%%%%%%%%%%%%%%%%%%%%%%%%%%%%%%%%%%%%%%%%%%%%%%%%%%%%%%%%%%
\begin{rem} \label{remclockH}
We highlight the different behaviour of the clock $\mathcal{H}$ for $t\rightarrow0$ and for $t\rightarrow\infty$ (for the second see e.g. \cite{PiY84}, followed by \cite{PiY86,LGY86,LGY87}, a result which is equivalent to Spitzer's Theorem \cite{Spi58} stated in (\ref{Spi}) ), that is:
\beq
\frac{\mathcal{H}(t)}{t}&\overset{{a.s.}}{\underset{t\rightarrow0}\longrightarrow}& 1 \ , \\
\frac{4 \mathcal{H}(t)}{(\log t)^{2}}  &\overset{{(law)}}{\underset{t\rightarrow\infty}\longrightarrow}& T_{1} \equiv \inf \{ t: \beta_{t}=1 \} \ , \label{HBeslarge}
\eeq
where the latter follows essentially from the classical Laplace argument: $$\|\cdot\|_{p}{\overset{p\rightarrow \infty}\longrightarrow}\|\cdot\|_{\infty} \ .$$
We also remark that, from Remark \ref{remclocklarge}, the behaviour of the clock for $t\rightarrow0$ is a.s. the same for Brownian motion and for stable processes, whereas it is different for $t\rightarrow\infty$. In particular, for $t\rightarrow\infty$, compare (\ref{clocktlargebis}) to
(\ref{HBeslarge}).
\end{rem}
%%%%%%%%%%%%%%%%%%%%%%%%%%%%%%%%%%%%%%%%%%%%%%%%%%%%%%%%%%%%%%%%%%
\vspace{5pt}
%%%%%%%%%%%%%%%%%%%%%%%%%%%%%%%%%%%%%%%%%%%%%%%%%%%%%%%%%%%%%%%%
\noindent \textbf{Acknowledgements} \\
The author S. Vakeroudis is very grateful to Professor M. Yor for the financial support
during his stay at the University of Manchester as a Post Doc fellow invited by Professor R.A. Doney.
%%%%%%%%%%%%%%%%%%%%%%%%%%%%%%%%%%%%%%%%%%%%%%%%%%%%%%%%%%%%%%%%

%%%%%%%%%%%%%%%%%%%%%%%%%%%%%%%%%%%%%%%%%%%%%%%%%%%%%%%%%%%%%%%%

\end{document}